\documentclass{amsart}
\usepackage{amsmath}
\usepackage{amsfonts}

\newtheorem{theorem}{Theorem}
\theoremstyle{plain}

\newtheorem{corollary}{Corollary}

\newtheorem{definition}{Definition}
\newtheorem{example}{Example}

\numberwithin{equation}{section}

\begin{document}
\title[On Invariant Submanifolds of a generalized Kenmotsu manifold]{On
Invariant Submanifolds of a generalized Kenmotsu manifold}
\author{Aysel TURGUT\ VANLI}
\address{Departmans Matematic, Universty of Gazi,\\
06500, Ankara TURKEY}
\email{avanli@gazi.edu.tr}
\author{Ramazan SARI}
\email{ramazansr@gmail.com}

\begin{abstract}
In this paper, invariant submanifolds of a generalized Kenmotsu manifold are
studied. Necessary and sufficient conditions are given on a submanifold of a
generalized Kenmotsu manifold to be an invariant submanifold.In this case, we
investigate further properties of invariant submanifolds of \ a generalized
Kenmotsu manifold. In addition, some theorems are given related to an
invariant submanifold of a generalized Kenmotsu manifold.We consider
semiparallel and 2-semiparallel invariant submanifolds of a generalized
Kenmotsu manifold.
\end{abstract}

\maketitle

\section{ Introduction}

In 1963, Yano \cite{Y} introduced an $f$-structure on a C$^{\infty }$
m-dimensional manifold $M$, defined by a non-vanishing tensor field $\varphi 
$ of type $(1,1)$ which satisfies $\varphi ^{3}+\varphi =0$ and has constant
rank $r$.It is know that in this case $r$ is even, $r=2n.$ Moreover, $TM$ splits into two complementary subbundles $Im\varphi $ and $\ker\varphi $ and the restriction of $\varphi $ to $ Im\varphi $
determines a complex structure on such subbundle. It is known that the
existence of an $f$-structure on $M$ is equivalent to a reduction of the
structure group to $U(n)\times O(s)$ \cite{B}, where $s=m-2n$.

In \cite{K} , K. Kenmotsu has introduced a Kenmotsu manifold. In \cite{TVS},
present autors have introduced generalized Kenmotsu manifolds.

In \cite{D} , J. Deprez defined a semi-parallel immersion. Let $%
i:M\rightarrow \bar{M}$ be an isometrik immersion of a Riemannian manifold, $%
\bar{R}$ curvature tensor of the Van der Waerden-Bortolotti connection $\bar{%
\nabla}$ of $i$ and $h$ the second fundemental from of $i,$ then $i$ is said
to be semi-parallel if $\bar{R}.h=0.$In addition, J. Deprez studied
semi-parallel hypersurfaces in \cite{De}.

In \cite{ALMO} , K. Arslan and colleagues defined, a submanifold to be
2-semi-parallel if $R.\nabla h=0$ where R curvature tensor of the Van der
Waerden-Bortolotti connection $\nabla $ of $i$ and h the second fundemental
from of $i.$

In this paper, we give necessary and sufficient conditions for a submanifold
of a generalized Kenmotsu manifold to be an invariant submanifold and we
consider the invariant case. In addition, we study semi-parallel and
2-semi-parallel invariant submanifolds of a generalized Kenmotsu manifold.

\section{Generalized Kenmotsu Manifolds}

\bigskip In \textit{\cite{GY},} a $(2n+s)-$dimensional differentiable
manifold $M$ is called metric $f-$manifold if there exist an $(1,1)$ type
tensor field $\varphi $, $s$ vector fields $\xi _{1},\dots ,\xi _{s}$, $s$ $%
1 $-forms $\eta ^{1},\dots ,\eta ^{s}$ and a Riemannian metric $g$ on $M$
such that%
\begin{equation}
\varphi ^{2}=-I+\overset{s}{\underset{i=1}{\sum }}\eta ^{i}\otimes \xi _{i},%
\begin{array}{cc}
\begin{array}{c}
\end{array}
& 
\end{array}%
\eta ^{i}(\xi _{j})=\delta _{ij}
\end{equation}%
\begin{equation}
g(\varphi X,\varphi Y)=g(X,Y)-\overset{s}{\underset{i=1}{\sum }}\eta
^{i}(X)\eta ^{i}(Y),
\end{equation}%
for any $X,Y\in \Gamma (TM),$ $i,j\in \{1,\dots ,s\}$. In addition, we have%
\begin{equation}
\eta ^{i}(X)=g(X,\xi _{i}),%
\begin{array}{cc}
\begin{array}{c}
\end{array}
& 
\end{array}%
g(X,\varphi Y)=-g(\varphi X,Y),%
\begin{array}{cc}
& 
\end{array}%
\varphi \xi _{i}=0,%
\begin{array}{cc}
& 
\end{array}%
\eta ^{i}\circ \varphi =0.
\end{equation}

Then, a $2$-form $\Phi $ is defined by $\Phi (X,Y)=g(X,\varphi Y)$, for any $%
X,Y\in \Gamma (TM)$, called the \textit{fundamental }$\mathit{2}$\textit{%
-form}. Moreover, a framed metric manifold is \textit{normal} if 
\begin{equation*}
\lbrack \varphi ,\varphi ]+2\underset{i=1}{\overset{s}{\sum }}d\eta
^{i}\otimes \xi _{i}=0,
\end{equation*}%
where $[\varphi ,\varphi ]$ is denoting the Nijenhuis tensor field
associated to $\varphi $ .

In \cite{van}, let M be a (2n+s)-dimensional metric f-manifold. If there exists
2-form $\Phi $ such that $\eta ^{1}\wedge ...\wedge \eta ^{s}\wedge \Phi
^{n}\neq 0$ on M then M is called an almost s-contact metric structure.

\begin{definition}
Let $M$ be an almost $s-$contact metric manifold of dimension $(2n+s)$, $%
s\geq 1$, with $\left( \varphi ,\xi _{i},\eta ^{i},g\right) $ . $M$ is said
to be a generalized almost Kenmotsu manifold if for all $1\leq i\leq s,$ $1-$%
forms $\eta ^{i}$ are closed and $d\Phi =2\underset{i=1}{\overset{s}{\sum }}%
\eta ^{i}\wedge \Phi .$ A normal generalized almost Kenmotsu manifold $M$ is
called a generalized Kenmotsu manifold \cite{TVS}.
\end{definition}

\begin{example}
Let $n=2$ and $s=3$. The vector fields%
\begin{equation*}
e_{1}=f_{1}(z_{1},z_{2},z_{3})\frac{\partial }{\partial x_{1}}%
+f_{2}(z_{1},z_{2},z_{3})\frac{\partial }{\partial y_{1}},
\end{equation*}%
\begin{equation*}
e_{2}=-f_{2}(z_{1},z_{2},z_{3})\frac{\partial }{\partial x_{1}}%
+f_{1}(z_{1},z_{2},z_{3})\frac{\partial }{\partial y_{1}},
\end{equation*}%
\begin{equation*}
e_{3}=f_{1}(z_{1},z_{2},z_{3})\frac{\partial }{\partial x_{2}}%
+f_{2}(z_{1},z_{2},z_{3})\frac{\partial }{\partial y_{2}},
\end{equation*}%
\begin{equation*}
e_{4}=-f_{2}(z_{1},z_{2},z_{3})\frac{\partial }{\partial x_{2}}%
+f_{1}(z_{1},z_{2},z_{3})\frac{\partial }{\partial y_{2}},
\end{equation*}%
\begin{equation*}
e_{5}=\frac{\partial }{\partial z_{1}},e_{6}=\frac{\partial }{\partial z_{2}}%
,e_{7}=\frac{\partial }{\partial z_{3}}
\end{equation*}%
where $f_{1}$ and $f_{2}$ are given by 
\begin{eqnarray*}
f_{1}(z_{1},z_{2},z_{3}) &=&c_{2}e^{-(z_{1}+z_{2}+z_{3})}\cos
(z_{1}+z_{2}+z_{3})-c_{1}e^{-(z_{1}+z_{2}+z_{3})}\sin (z_{1}+z_{2}+z_{3}), \\
f_{2}(z_{1},z_{2},z_{3}) &=&c_{1}e^{-(z_{1}+z_{2}+z_{3})}\cos
(z_{1}+z_{2}+z_{3})+c_{2}e^{-(z_{1}+z_{2}+z_{3})}\sin (z_{1}+z_{2}+z_{3})
\end{eqnarray*}%
for nonzero constant $c_{1},c_{2}.$ It is obvious that $\left\{
e_{1},e_{2},e_{3},e_{4},e_{5},e_{6},e_{7}\right\} $ are linearly independent
at each point of $M$. Let $g$ be the Riemannian metric defined by%
\begin{equation*}
g(e_{i},e_{j})=\left\{ 
\begin{array}{lll}
1, &  & \text{for }i=j \\ 
0, &  & \text{for }i\neq j%
\end{array}%
\right.
\end{equation*}%
for all $i,j\in \left\{ 1,2,3,4,5,6,7\right\} $ and given by the tensor
product%
\begin{equation*}
g=\frac{1}{f_{1}^{2}+f_{2}^{2}}\sum_{i=1}^{2}(dx_{i}\otimes
dx_{i}+dy_{i}\otimes dy_{i})+dz_{1}\otimes dz_{1}+dz_{2}\otimes
dz_{2}+dz_{3}\otimes dz_{3},
\end{equation*}%
where $\left\{ x_{1},y_{1},x_{2},y_{2},z_{1},z_{2},z_{3}\right\} $ are
standard coordinates in $%
\mathbb{R}
^{7}$. Let $\eta ^{1}$, $\eta ^{2}$ and $\eta ^{3}$ be the $1$-form defined
by $\eta ^{1}(X)=g(X,e_{5})$, $\eta ^{2}(X)=g(X,e_{6})$ and $\eta
^{3}(X)=g(X,e_{7})$, respectively, for any vector field $X$ on $M$ and $\varphi 
$ be the $(1,1)$ tensor field defined by 
\begin{eqnarray*}
\varphi (e_{1}) &=&e_{2},%
\begin{array}{cc}
\begin{array}{c}
\end{array}
& 
\end{array}%
\varphi (e_{2})=-e_{1},%
\begin{array}{cc}
& 
\end{array}%
\varphi (e_{3})=e_{4},%
\begin{array}{cc}
\begin{array}{c}
\end{array}
& 
\end{array}%
\varphi (e_{4})=-e_{3}, \\
\varphi (e_{5} &=&\xi _{1})=0,%
\begin{array}{cc}
\begin{array}{c}
\end{array}
& 
\end{array}%
\varphi (e_{6}=\xi _{2})=0,%
\begin{array}{cc}
\begin{array}{c}
\end{array}
& 
\end{array}%
\varphi (e_{7}=\xi _{3})=0.
\end{eqnarray*}%
Therefore, the essential non-zero component of $\Phi $ is%
\begin{equation*}
\Phi (\frac{\partial }{\partial x_{i}},\frac{\partial }{\partial y_{i}})=-%
\frac{1}{f_{1}^{2}+f_{2}^{2}}=-\frac{2e^{2(z_{1}+z_{2}+z_{3})}}{%
c_{1}^{2}+c_{2}^{2}},%
\begin{array}{cc}
& 
\end{array}%
i=1,2
\end{equation*}%
and hence 
\begin{equation*}
\Phi =-\frac{2e^{2(z_{1}+z_{2}+z_{3})}}{c_{1}^{2}+c_{2}^{2}}%
\sum_{i=1}^{2}dx_{i}\wedge dy_{i}.
\end{equation*}%
Consequently, the exterior derivative $d\Phi $ is given by 
\begin{equation*}
d\Phi =-\frac{4e^{2(z_{1}+z_{2}+z_{3})}}{c_{1}^{2}+c_{2}^{2}}%
(dz_{1}+dz_{2}+dz_{3})\wedge \sum_{i=1}^{2}dx_{i}\wedge dy_{i}.
\end{equation*}%
Since $\eta ^{1}=dz_{1}$, $\eta ^{2}=dz_{2}$ and $\eta ^{3}=dz_{3},$ we find%
\begin{equation*}
d\Phi =2(\eta ^{1}+\eta ^{2}+\eta ^{3})\wedge \Phi .
\end{equation*}%
In addition, Nijenhuis tersion of $\varphi $ is different from zero.
\end{example}

\begin{theorem}
Let $\left( M,\varphi ,\xi _{i},\eta ^{i},g\right) $ be an almost $s$%
-contact metric manifold. $M$ is a generalized Kenmotsu manifold if and only
if%
\begin{equation}
\left( \bar{\nabla}_{X}\varphi \right) Y=\underset{i=1}{\overset{s}{\sum }}%
\left\{ g(\varphi X,Y)\xi _{i}-\eta ^{i}(Y)\varphi X\right\}
\end{equation}%
for all $X,Y\in \Gamma (TM),$ $i\in \left\{ 1,2,...,s\right\} ,$ where $\bar{%
\nabla}$ is Riemannian connection on M \cite{TVS}.
\end{theorem}

\begin{theorem}
Let M be a $(2n+s)$-dimensional generalized Kenmotsu manifold with structure 
$\left( \varphi ,\xi _{i},\eta ^{i},g\right) .$ Then we have%
\begin{equation}
\bar{\nabla}_{X}\xi _{j}=-\varphi ^{2}X
\end{equation}%
for all $X,Y\in \Gamma (TM),i\in \left\{ 1,2,...,s\right\} $ \ \cite{TVS}.
\end{theorem}

\begin{theorem}
Let M be a $(2n+s)$-dimensional generalized Kenmotsu manifold with structure 
$\left( \varphi ,\xi _{i},\eta ^{i},g\right) .$ Then we have%
\begin{equation}
\bar{R}(X,Y)\xi _{i}=\underset{j=1}{\overset{s}{\sum }}\{\eta ^{j}(Y)\varphi
^{2}X-\eta ^{j}(X)\varphi ^{2}Y\}
\end{equation}%
\begin{equation}
\bar{R}(X,\xi _{j})\xi _{i}=\varphi ^{2}X
\end{equation}%
\begin{equation}
\bar{R}(\xi _{j},X)\xi _{i}=-\varphi ^{2}X
\end{equation}%
\begin{equation}
\bar{R}(\xi _{k},\xi _{j})\xi _{i}=0
\end{equation}%
\begin{equation}
\bar{R}(\xi _{j},X)Y=\underset{j=1}{\overset{s}{\sum }}\{g(X,\varphi
^{2}Y)\xi _{j}-\eta ^{j}(Y)\varphi ^{2}X\}
\end{equation}%
for all $X,Y\in \Gamma (TM),i,j,k\in \left\{ 1,2,...,s\right\} $ \cite{TVS}.
\end{theorem}

\section{\textbf{\ }Invariant Submanifolds of a Generalized Kenmotsu Manifold}

Let $M$ be a submanifold of the a $(2n+s)$ dimensional generalized Kenmotsu
manifold $\bar{M}$. If $\varphi (T_{x}M))\subset T_{x}M$, for any point $%
x\in M$ and $\xi _{i}$ are tanget to M for all $i\in \{1,2,...,s\},$ then $M$
is called an invariant submanifold of $\bar{M}$.

Let $\nabla $ be the Levi-Civita connection of $M$ with respect to the
induced metric g. Then Gauss and Weingarten formulas are given by%
\begin{equation}
\bar{\nabla}_{X}Y=\nabla _{X}Y+h(X,Y)
\end{equation}%
\begin{equation}
\bar{\nabla}_{X}N=\nabla _{X}^{\perp }N-A_{N}X
\end{equation}%
for any $X,Y\in \Gamma (TM)$ and $N\in \Gamma (TM)^{\perp }$. $\nabla
^{\perp }$ is the connection in the normal bundle, h is the second
fundamental from of M and $A_{N}$ is the Weingarten endomorphism associated
with N. The second fundamental form h and the shape operator A related by%
\begin{equation}
g(h(X,Y),N)=g(A_{N}X,Y).
\end{equation}%
The curvature transformations of $M$ and $\bar{M}$ will be denote by%
\begin{equation}
R(X,Y)=\nabla _{X}\nabla _{Y}-\nabla _{Y}\nabla _{X}-\nabla _{\left[ X,Y%
\right] }\ \ \text{\ }X,Y\in \Gamma (TM)
\end{equation}%
and%
\begin{equation}
\bar{R}(X,Y)=\bar{\nabla}_{X}\bar{\nabla}_{Y}-\bar{\nabla}_{Y}\bar{\nabla}%
_{X}-\bar{\nabla}_{\left[ X,Y\right] }\ \ \ X,Y\in \Gamma (T\bar{M})
\end{equation}%
respectively.

Using $(3,4)$, $(3,5),$ the Gauss and the Weingarten formulas, we obtain for
any vector fields $X$,$Y$ and $Z$ tanget to $M$%
\begin{eqnarray}
\bar{R}(X,Y)Z &=&R(X,Y)Z-A_{h(Y,Z)}(X)+A_{h(X,Z)}(Y)  \notag \\
&&+(\nabla _{X}h)(Y,Z)-(\nabla _{Y}h)(X,Z).
\end{eqnarray}%
Thus, if $W$ is tangent to $M$, then we get the Gauss equation%
\begin{eqnarray}
\bar{g}(\bar{R}(X,Y)Z,W) &=&g(R(X,Y)Z,W)+\bar{g}(h(Y,W),h(X,Z))  \notag \\
&&-\bar{g}(h(X,W),h(Y,Z)).
\end{eqnarray}

\begin{theorem}
Let $M$ be an invariant submanifold of the a $(2n+s)$-dimensional 
generalized Kenmotsu manifold $\bar{M}.$Then we have,%
\begin{equation}
\left( \nabla _{X}\varphi \right) Y=\underset{i=1}{\overset{s}{\sum }}%
\left\{ g(\varphi X,Y)\xi _{i}-\eta ^{i}(Y)\varphi X\right\} ,
\end{equation}%
\begin{equation}
h(X,\varphi Y)=\varphi h(X,Y)
\end{equation}%
for all $X,Y\in \Gamma (TM).$
\end{theorem}

\begin{proof}
Since $\bar{M}$ is a generalized Kenmotsu manifold, we have%
\begin{equation*}
\left( \bar{\nabla}_{X}\varphi \right) Y=\underset{i=1}{\overset{s}{\sum }}%
\left\{ g(\varphi X,Y)\xi _{i}-\eta ^{i}(Y)\varphi X\right\} .
\end{equation*}

Using $(3,1)$ then we have,%
\begin{eqnarray*}
(\bar{\nabla}_{X}\varphi )Y &=&\nabla _{X}\varphi Y-h(X,\varphi Y)-\varphi
(\nabla _{X}Y-h(X,Y)) \\
&=&(\nabla _{X}\varphi )Y-h(X,\varphi Y)+\varphi h(X,Y).
\end{eqnarray*}%
Comparing the tangential and normal part of last equation, we get the
desired result.
\end{proof}

\begin{corollary}
Let $M$ be an invariant submanifold of the a $(2n+s)$-dimensional 
generalized Kenmotsu manifold $\bar{M}.$Then we have,%
\begin{equation}
h(X,\varphi Y)=\varphi h(X,Y)=h(\varphi X,Y),
\end{equation}%
\begin{equation}
h(\varphi X,\varphi Y)=-h(X,Y)
\end{equation}%
for all $X,Y\in \Gamma (TM).$
\end{corollary}

\begin{theorem}
Let $M$ be an invariant submanifold of the a $(2n+s)$-dimensional 
generalized Kenmotsu manifold $\bar{M}.$Then we have,%
\begin{equation}
\nabla _{X}\xi _{j}=-\varphi ^{2}X,
\end{equation}%
\begin{equation}
h(X,\xi _{j})=0
\end{equation}%
for all $X,Y\in \Gamma (TM).$
\end{theorem}

\begin{proof}
For a generalized Kenmotsu manifold using $(2,5)$, then we get%
\begin{equation*}
\overline{\nabla }_{X}\xi _{j}=X-\overset{s}{\underset{i=1}{\sum }}\eta
^{i}(X)\xi _{i}
\end{equation*}%
\begin{equation*}
\nabla _{X}\xi _{j}+h(X,\xi _{j})=X-\overset{s}{\underset{i=1}{\sum }}\eta
^{i}(X)\xi _{i}
\end{equation*}%
from the Gauss formula then the equation is implied that\ 
\begin{equation*}
\nabla _{X}\xi _{j}=X-\overset{s}{\underset{i=1}{\sum }}\eta ^{i}(X)\xi _{i}
\end{equation*}%
and%
\begin{equation*}
h(X,\xi _{j})=0.
\end{equation*}
\end{proof}

\begin{theorem}
Let $M$ be an invariant submanifold of the a $(2n+s)$-dimensional 
generalized Kenmotsu manifold $\bar{M}.$Then we have,%
\begin{equation}
R(X,Y)\xi _{i}=\underset{j=1}{\overset{s}{\sum }}\{\eta ^{j}(Y)\varphi
^{2}X-\eta ^{j}(X)\varphi ^{2}Y\}
\end{equation}%
for all $X,Y\in \Gamma (TM).$
\end{theorem}

\begin{proof}
Using Gauss equation $(3,6)$, we have%
\begin{eqnarray*}
\bar{R}(X,Y)\xi _{i} &=&R(X,Y)\xi _{i}-A_{h(Y,\xi _{i})}(X)+A_{h(X,\xi
_{i})}(Y)+(\nabla _{X}h)(Y,\xi _{i})-(\nabla _{Y}h)(X,\xi _{i}). \\
&=&R(X,Y)\xi _{i}-A_{h(Y,\xi _{i})}(X)+A_{h(X,\xi _{i})}(Y) \\
&&+\nabla _{X}h(Y,\xi _{i})-h(\nabla _{X}Y,\xi _{i})-h(Y,\nabla _{X}\xi _{i})
\\
&&-\nabla _{Y}h(X,\xi _{i})+h(\nabla _{Y}X,\xi _{i})+h(X,\nabla _{Y}\xi
_{i}).
\end{eqnarray*}%
From $(3,10)$, $(3,12)$ and $(3,13)$, we get%
\begin{equation*}
\bar{R}(X,Y)\xi _{i}=R(X,Y)\xi _{i}.
\end{equation*}
\end{proof}

\begin{corollary}
Let $M$ be an invariant submanifold of the a $(2n+s)$-dimensional 
generalized Kenmotsu manifold $\bar{M}.$Then we have,%
\begin{equation}
R(\xi _{j},X)\xi _{i}=-\varphi ^{2}X
\end{equation}%
\begin{equation*}
R(X,\xi _{j})\xi _{i}=\varphi ^{2}X
\end{equation*}%
\begin{equation*}
R(\xi _{k},\xi _{j})\xi _{i}=0
\end{equation*}%
\begin{equation}
R(\xi _{j},X)Y=\underset{j=1}{\overset{s}{\sum }}\{g(X,\varphi ^{2}Y)\xi
_{j}-\eta ^{j}(Y)\varphi ^{2}X\}
\end{equation}%
for all $X,Y\in \Gamma (TM),i,j,k\in \left\{ 1,2,...,s\right\} $.
\end{corollary}

\begin{corollary}
An invariant submanifold $M$ of a generalized Kenmotsu manifold $\bar{M},$ $%
\xi _{i}$ are tanget to $M$ for all $i\in \{1,2,...,s\},$ is a generalized
Kenmotsu manifold.
\end{corollary}

\begin{example}
The consider a submanifold of example1 defined by%
\begin{equation*}
M=X(u,v,w_{1},w_{2},w_{3})=(f_{1}u-f_{2}v,f_{1}u+f_{2}v,0,0,w_{1},w_{2},w_{3})
\end{equation*}%
where $f_{1}$ and $f_{2}$ are given by%
\begin{eqnarray*}
f_{1}(z_{1},z_{2},z_{3}) &=&c_{2}e^{-(z_{1}+z_{2}+z_{3})}\cos
(z_{1}+z_{2}+z_{3})-c_{1}e^{-(z_{1}+z_{2}+z_{3})}\sin (z_{1}+z_{2}+z_{3}), \\
f_{2}(z_{1},z_{2},z_{3}) &=&c_{1}e^{-(z_{1}+z_{2}+z_{3})}\cos
(z_{1}+z_{2}+z_{3})+c_{2}e^{-(z_{1}+z_{2}+z_{3})}\sin (z_{1}+z_{2}+z_{3})
\end{eqnarray*}%
for nonzero constant $c_{1},c_{2}.$ Then local frame of TM%
\begin{eqnarray*}
e_{1} &=&f_{1}(z_{1}+z_{2})\frac{\partial }{\partial x}+f_{2}(z_{1}+z_{2})%
\frac{\partial }{\partial y} \\
e_{2} &=&-f_{2}(z_{1}+z_{2})\frac{\partial }{\partial x}+f_{1}(z_{1}+z_{2})%
\frac{\partial }{\partial y} \\
e_{3} &=&\frac{\partial }{\partial z_{1}},%
\begin{array}{cc}
& 
\end{array}%
e_{4}=\frac{\partial }{\partial z_{2}},%
\begin{array}{cc}
& 
\end{array}%
e_{5}=\frac{\partial }{\partial z_{3}}
\end{eqnarray*}%
We can easily that M is an invariant submanifold.
\end{example}

\begin{theorem}
Let $M$ be an invariant submanifold of a $(2n+s)$-dimensional generalized
Kenmotsu manifold $\bar{M}.$Then we have,%
\begin{equation}
(\nabla _{X}h)(Y,\xi _{i})=-h(Y,\nabla _{X}\xi _{i})
\end{equation}%
for all $X,Y\in \Gamma (TM).$
\end{theorem}

\begin{proof}
Using $(3,13),$ then we have%
\begin{equation*}
(\nabla _{X}h)(Y,\xi _{i})=\nabla _{X}h(Y,\xi _{i})-h(\nabla _{X}Y,\xi
_{i})-h(Y,\nabla _{X}\xi _{i})=-h(Y,\nabla _{X}\xi _{i}).
\end{equation*}
\end{proof}

\begin{corollary}
Let $M$ be an invariant submanifold of a $(2n+s)$-dimensional generalized
Kenmotsu manifold $\bar{M}.$Then we have,%
\begin{equation}
(\nabla _{X}h)(Y,\xi _{i})=-h(X,Y)
\end{equation}%
for all $X,Y\in \Gamma (TM).$
\end{corollary}

\begin{theorem}
Let $M$ be an invariant submanifold of a $(2n+s)$-dimensional  generalized
Kenmotsu manifold $\bar{M}.$Then $h$ is parallel if and only if $M$ is
totally geodesic.
\end{theorem}

\begin{proof}
Suppose that $h$ is parallel. Then, for each \ $X,Y\in \Gamma (TM)$,$\qquad
\qquad \qquad \ \ \ \ \ \ $%
\begin{equation*}
\ (\nabla _{X}h)(Y,\xi _{i})=0.
\end{equation*}%
Using $(3,18),$ we get 
\begin{equation*}
h(X,Y)=0.
\end{equation*}%
Vice versa, let $M$ be totally geodesic. Then $h=0.$ For all $X,Y,Z\in
\Gamma (TM),$ 
\begin{equation*}
(\nabla _{X}h)(Y,Z)=\nabla _{X}h(Y,Z)-h(\nabla _{X}Y,Z)-h(Y,\nabla _{X}Z)=0.
\end{equation*}%
Thus we have%
\begin{equation*}
\nabla h=0.
\end{equation*}
\end{proof}

\begin{theorem}
Let $M$ be an invariant submanifold of a $(2n+s)$-dimensional generalized
Kenmotsu manifold $\bar{M}.$Then we have,%
\begin{equation*}
A_{N}\xi _{i}=0
\end{equation*}%
for all $N\in \Gamma (TM^{\perp }).$
\end{theorem}

\begin{proof}
Using $(3,3)$ and $(3,13)$, we get $g(A_{N}\xi _{i},Y)=g(h(\xi _{i},Y),N)=0.$
\end{proof}

\begin{theorem}
Let $M$ be an invariant submanifold of a $(2n+s)$-dimensional generalized
Kenmotsu manifold $\bar{M}.$Then we have,%
\begin{equation*}
\varphi (A_{N}X)=A_{\varphi N}X=-A_{N}\varphi X
\end{equation*}%
for all $X\in \Gamma (TM),$ $N\in \Gamma (TM^{\perp }).$
\end{theorem}

\begin{proof}
By using $(2,3),(3,3)$ and $(3,10)$ for all $X\in \Gamma (TM),$ $N\in \Gamma
(TM)^{\perp }$ we get%
\begin{eqnarray*}
\ \ g(\varphi (A_{N}X),Y) &=&-g(A_{N}X,\varphi Y) \\
&=&-\ g(h(X,\varphi Y),N) \\
\ &=&-\ g(h(\varphi X,Y),N) \\
\ \ &=&-\ g(\ A_{N}\varphi X,Y).
\end{eqnarray*}%
Then, we have $\ \ \varphi (A_{N}X)=-A_{N}\varphi X.$

Moreover,%
\begin{eqnarray*}
g(A_{\varphi N}X,Y) &=&g(h(X,Y),\varphi N) \\
&=&-g(\varphi (h(X,Y)),N) \\
&=&-g(h(X,\varphi Y),N) \\
&=&-g(A_{N}X,\varphi Y) \\
&=&g(\varphi (A_{N}X),Y).
\end{eqnarray*}%
Then, we get $\ \ A_{\varphi N}X=\varphi (A_{N}X).$
\end{proof}

\begin{theorem}
Let $M$ be an invariant submanifold of a $(2n+s)$-dimensional generalized
Kenmotsu manifold $\bar{M}.$Then, the second fundemental form $h$ is $\eta $%
-parallel if and only if 
\begin{equation*}
(\nabla _{X}h)(Y,Z)=-\overset{s}{\underset{i=1}{\sum }}\{\eta
^{i}(Y)h(X,Z)+\eta ^{i}(Z)h(X,Y)\}
\end{equation*}%
for all $X,Y,Z\in \Gamma (TM).$
\end{theorem}

\begin{proof}
\textit{\ }Let $M$ be an invariant submanifold of a generalized Kenmotsu
manifold $\bar{M}$. The second fundemantal form $h$ of $M$ is said to be $%
\eta -parallel$ if $(\nabla _{X}h)(\varphi Y,\varphi Z)=0$ for all vector
fields $X,Y$ and $Z$ tangent to $M$.

First of all, we have%
\begin{equation*}
(\nabla _{X}h)(\varphi Y,\varphi Z)=\nabla _{X}h(\varphi Y,\varphi
Z)-h(\nabla _{X}\varphi Y,\varphi Z)-h(\varphi Y,\nabla _{X}\varphi Z).
\end{equation*}%
Then%
\begin{equation*}
\nabla _{X}h(\varphi Y,\varphi Z)=h(\nabla _{X}\varphi Y,\varphi
Z)+h(\varphi Y,\nabla _{X}\varphi Z).
\end{equation*}%
Using $(3,11),$ we get%
\begin{eqnarray*}
-\nabla _{X}h(Y,Z) &=&h((\nabla _{X}\varphi )Y+\varphi (\nabla
_{X}Y),\varphi Z)+h(\varphi Y,(\nabla _{X}\varphi )Z+\varphi (\nabla _{X}Z))
\\
&=&h((\nabla _{X}\varphi )Y,\varphi Z)-h(\nabla _{X}Y,Z)+h(\varphi Y,(\nabla
_{X}\varphi )Z)-h(Y,\nabla _{X}Z).
\end{eqnarray*}%
Thus, by using $(3,9)$ we have%
\begin{equation*}
-\nabla _{X}h(Y,Z)=-\overset{s}{\underset{i=1}{\sum }}\eta ^{i}(Y)h(\varphi
X,\varphi Z)-h(\nabla _{X}Y,Z)-\overset{s}{\underset{i=1}{\sum }}\eta
^{i}(Z)h(\varphi Y,\varphi X)-h(Y,\nabla _{X}Z).
\end{equation*}
\end{proof}

\begin{theorem}
Let $M$ be an invariant submanifold of a $(2n+s)$-dimensional  generalized
Kenmotsu manifold $\bar{M}.$Then $\bar{R}(X,Y)\xi _{j}$ is tangent to $M$
for any $X,Y\in \Gamma (TM).$
\end{theorem}

\begin{proof}
For each $N_{l}\in \Gamma (TM)^{\perp }$ we have 
\begin{eqnarray*}
\bar{g}(\bar{R}(X,Y)\xi _{j},N_{l}) &=&\bar{g}(\overset{s}{\underset{i=1}{%
\sum }}\{\eta ^{i}(Y)\varphi ^{2}X-\eta ^{i}(X)\varphi ^{2}Y\},N_{l}) \\
&=&\bar{g}(\overset{s}{\underset{i=1}{\sum }}\eta ^{i}(Y)\varphi
^{2}X,N_{l})+\bar{g}(\overset{s}{\underset{i=1}{\sum }}\eta ^{i}(X)\varphi
^{2}Y,N_{l}) \\
&=&\overset{s}{\underset{i=1}{\sum }}\eta ^{i}(Y)\{\bar{g}(-X+\overset{s}{%
\underset{k=1}{\sum }}\eta ^{k}(X)\xi _{k},N_{l})\} \\
&&+\overset{s}{\underset{i=1}{\sum }}\eta ^{i}(X)\{\bar{g}(-Y+\overset{s}{%
\underset{k=1}{\sum }}\eta ^{k}(Y)\xi _{k},N_{l})\} \\
&=&\overset{s}{\underset{i,k=1}{\sum }}\{-\bar{g}(X,N_{l})+\eta ^{i}(Y)\eta
^{k}(X)\bar{g}(\xi _{k},N_{l}) \\
&&\text{ \ \ \ \ \ \ \ }-\bar{g}(Y,N_{l})+\eta ^{i}(Y)\eta ^{k}(X)\bar{g}%
(\xi _{k},N_{l})\} \\
\ &=&0.
\end{eqnarray*}
\end{proof}

\begin{theorem}
Let $M$ be an invariant submanifold of a $(2n+s)$-dimensional generalized
Kenmotsu manifold $\bar{M}.$Then we have,%
\begin{equation*}
g(R(X,\varphi X)\varphi X,X)=\bar{g}(\bar{R}(X,\varphi X)\varphi X,X)-2\bar{g%
}(h(X,X),h(X,X)).
\end{equation*}
\end{theorem}

\begin{proof}
First of all, we have 
\begin{eqnarray*}
\bar{g}(\bar{R}(X,\varphi X)\varphi X,X) &=&g(R(X,\varphi X)\varphi X,X)-%
\bar{g}(h(X,X),h(\varphi X,\varphi X)) \\
&&+\bar{g}(h(\varphi X,X),h(X,\varphi X))
\end{eqnarray*}%
From $(3,10)$, we get $\qquad $%
\begin{eqnarray*}
g(R(X,\varphi X)\varphi X,X) &=&\bar{g}(\bar{R}(X,\varphi X)\varphi X,X)+%
\bar{g}(h(X,X),h(\varphi X,\varphi X)) \\
&&-\bar{g}(\varphi h(X,X),\varphi h(X,X)).
\end{eqnarray*}%
Using $(3,11)$ we have $\qquad $%
\begin{equation*}
g(R(X,\varphi X)\varphi X,X)=\bar{g}(\bar{R}(X,\varphi X)\varphi X,X)-2\bar{g%
}(h(X,X),h(X,X)).
\end{equation*}
\end{proof}

\section{Semiparallel and 2-Semiparallel Invariant Submanifolds of A
Generalized Kenmotsu Manifold}

Let $M$ be a submanifold of a Riemannian manifold $\bar{M}.$An isometric
immersion $i:M\rightarrow \bar{M}$ is \textit{semi-parallel} if 
\begin{equation*}
\bar{R}(X,Y)h=\bar{\nabla}_{X}(\bar{\nabla}_{Y}h)-\bar{\nabla}_{Y}(\bar{%
\nabla}_{X}h)-\bar{\nabla}_{[X,Y]}h=0
\end{equation*}%
where $\bar{R}$ is the curvature tensor of $\bar{\nabla}$ \cite{C}, where $%
\bar{R}$ curvature tensor of the Van der Waerden-Bortolotti connection $\bar{%
\nabla}$ and $h$ the second fundemental from.

In \cite{ALMO}, K. Arslan and colleagues defined that $M$ is \textit{%
2-semiparallel} submanifolds if 
\begin{equation*}
R(X,Y)\nabla h=0
\end{equation*}%
for all vector fields $X,Y$ tanget to $M$.

$\bar{\nabla}$ is the connection in $TM\oplus TM^{^{\perp }}$ build with $%
\nabla $ and $\nabla ^{^{\perp }}$, where R (resp. $R^{\perp }$) denotes
curvature tensor of the connection $\nabla $ (resp. $\nabla ^{^{\perp }}$).
If $R^{\perp }$ denotes the curvature tensor of $\nabla ^{^{\perp }}$ then,%
\begin{equation}
(\bar{R}(X,Y)h)(Z,U)=R^{\perp }(X,Y)h(Z,U)-h(R(X,Y)Z,U)-h(Z,R(X,Y)U)
\end{equation}%
for all vector fields $X,Y,Z,U$ tanget to $M$ \cite{C}. In addition,%
\begin{eqnarray}
(\bar{R}(X,Y)\bar{\nabla}h)(Z,U,V) &=&R^{\perp }(X,Y)(\bar{\nabla}h)(Z,U,V)-(%
\bar{\nabla}h)(R(X,Y)Z,U,V)  \notag \\
&&-(\bar{\nabla}h)(Z,R(X,Y)U)-(\bar{\nabla}h)(Z,U,R(X,Y)V)
\end{eqnarray}%
or all vector fields $X,Y,Z,U,V$ tanget to $M$ where $(\bar{\nabla}%
h)(Z,U,V)=(\bar{\nabla}_{Z}h)(U,V)$ \cite{ALMO}.

\begin{theorem}
Let $M$ be an invariant submanifold of a generalized Kenmotsu manifold $\bar{%
M}.$Then $M$ is semi-parallel if and only if $M$ is total geodesic.
\end{theorem}

\begin{proof}
Suppose that $M$ is semi-parallel. Then, $\bar{R}(X,Y)h=0$ for each $X,Y\in
\Gamma (TM)$.Using $(4,1)$, we get%
\begin{equation*}
R^{\perp }(X,Y)h(Z,K)-h(R(X,Y)Z,K)-h(Z,R(X,Y)K)=0.
\end{equation*}%
We take $X=\xi _{i}$ \ and $K=\xi _{j}$ then, 
\begin{equation*}
R^{\perp }(\xi _{i},Y)h(Z,\xi _{j})-h(R(\xi _{i},Y)Z,\xi _{j})-h(Z,R(\xi
_{i},Y)\xi _{j})=0.
\end{equation*}%
From $(3,13)$%
\begin{equation*}
h(Z,R(\xi _{i},Y)\xi _{j})=0.
\end{equation*}%
Using $(3,15)$%
\begin{equation*}
h(Z,Y-\overset{s}{\underset{t=1}{\sum }}\eta _{t}(X)\xi _{t})=0.
\end{equation*}%
Then, we have%
\begin{equation*}
h(Z,Y)=0.
\end{equation*}
\end{proof}

\begin{theorem}
Let $M$ be an invariant submanifold of a generalized Kenmotsu manifold $\bar{%
M}.$Then $M$ is 2-semiparallel if and only if $M$ is total geodesic.
\end{theorem}

\begin{proof}
Suppose that $M$ is 2-semiparallel. Then, $\bar{R}(X,Y)\bar{\nabla}h=0$ for
each $X,Y,Z,U,V\in \Gamma (TM)$.Using $(4,2)$, we get%
\begin{equation*}
R^{\perp }(X,Y)(\bar{\nabla}h)(Z,U,V)-(\bar{\nabla}h)(R(X,Y)Z,U,V)-(\bar{%
\nabla}h)(Z,R(X,Y)U)-(\bar{\nabla}h)(Z,U,R(X,Y)V)=0.
\end{equation*}%
We take $X=\xi _{i}$ \ and $U=\xi _{j}$ then, we have%
\begin{equation*}
R^{\perp }(\xi _{i},Y)(\bar{\nabla}h)(Z,\xi _{j},V)-(\bar{\nabla}h)(R(\xi
_{i},Y)Z,\xi _{j},V)-(\bar{\nabla}h)(Z,R(\xi _{i},Y)\xi _{j},V)-(\bar{\nabla}%
h)(Z,\xi _{j},R(\xi _{i},Y)V)=0.
\end{equation*}%
From $(3,11),(3,12),(3,13),(3,15)$ and $(3,16)$ we get 
\begin{eqnarray*}
(\bar{\nabla}h)(Z,\xi _{j},V) &=&(\bar{\nabla}_{Z}h)(\xi _{j},V) \\
&=&\nabla _{Z}^{\perp }(h(\xi _{j},V))-h(\nabla _{Z}\xi _{j},V)-h(\xi
_{j},\nabla _{Z}V) \\
&=&-h(-\varphi ^{2}Z,V) \\
&=&-h(Z,V).
\end{eqnarray*}

Therefore, we have%
\begin{eqnarray*}
(\bar{\nabla}h)(R(\xi _{i},Y)Z,\xi _{j},V) &=&(\bar{\nabla}_{R(\xi
_{i},Y)Z}h)(\xi _{j},V) \\
&=&\nabla _{_{R(\xi _{i},Y)Z}}^{\perp }(h(\xi _{j},V))-h(\nabla _{R(\xi
_{i},Y)Z}\xi _{j},V)-h(\xi _{j},\nabla _{_{R(\xi _{i},Y)Z}}V) \\
&=&-h(-\varphi ^{2}R(\xi _{i},Y)Z,V) \\
&=&-h(R(\xi _{i},Y)Z,V) \\
&=&-h(\overset{s}{\underset{l=1}{\sum }}\{g(Y,\varphi ^{2}Z)\xi _{l}-\eta
^{l}(Z)\varphi ^{2}Y\},V) \\
&=&\overset{s}{\underset{l=1}{\sum }}\eta ^{l}(Z)h(\varphi ^{2}Y,V) \\
&=&-\overset{s}{\underset{l=1}{\sum }}\eta ^{l}(Z)h(Y,V),
\end{eqnarray*}%
\begin{eqnarray*}
(\bar{\nabla}h)(Z,R(\xi _{i},Y)\xi _{j},V) &=&(\bar{\nabla}_{Z}h)(R(\xi
_{i},Y)\xi _{j},V) \\
&=&\nabla _{Z}^{\perp }(h(R(\xi _{i},Y)\xi _{j},V))-h(\nabla _{Z}R(\xi
_{i},Y)\xi _{j},V)-h(R(\xi _{i},Y)\xi _{j},\nabla _{Z}V) \\
&=&\nabla _{Z}^{\perp }(h(-\varphi ^{2}Y,V))-h(\nabla _{Z}(-\varphi
^{2}Y),V)-h(-\varphi ^{2}Y,\nabla _{Z}V)
\end{eqnarray*}

and%
\begin{eqnarray*}
(\bar{\nabla}h)(Z,\xi _{j},R(\xi _{i},Y)V) &=&(\bar{\nabla}_{Z}h)(\xi
_{j},R(\xi _{i},Y)V) \\
&=&\nabla _{Z}^{\perp }(h(\xi _{j},R(\xi _{i},Y)V))-h(\nabla _{Z}\xi
_{j},R(\xi _{i},Y)V)-h(\xi _{j},\nabla _{Z}R(\xi _{i},Y)V) \\
&=&-h(\nabla _{Z}\xi _{j},R(\xi _{i},Y)V) \\
&=&-h(Z-\overset{s}{\underset{t=1}{\sum }}\eta ^{t}(Z)\xi _{t},R(\xi
_{i},Y)V) \\
&=&-h(Z,R(\xi _{i},Y)V) \\
&=&-h(Z,\overset{s}{\underset{l=1}{\sum }}\{g(Y,\varphi ^{2}V)\xi _{l}-\eta
^{l}(V)\varphi ^{2}Y\}) \\
&=&\overset{s}{\underset{l=1}{\sum }}\eta ^{l}(V)h(Z,\varphi ^{2}Y) \\
&=&-\overset{s}{\underset{l=1}{\sum }}\eta ^{l}(V)h(Z,Y).
\end{eqnarray*}%
Then, we get%
\begin{eqnarray*}
&&R^{\perp }(\xi _{i},Y)(-h(Z,V))-(-\overset{s}{\underset{l=1}{\sum }}\eta
^{l}(Z)h(Y,V))-(\nabla _{Z}^{\perp }(h(-\varphi ^{2}Y,V)) \\
&&-h(\nabla _{Z}(-\varphi ^{2}Y),V)-h(-\varphi ^{2}Y,\nabla _{Z}V))-(-%
\overset{s}{\underset{l=1}{\sum }}\eta ^{l}(V)h(Z,Y)) \\
&=&0.
\end{eqnarray*}%
So we take $V=\xi _{k}$ then, we have%
\begin{equation*}
h(\varphi ^{2}Y,\nabla _{Z}\xi _{k}))+\overset{s}{\underset{l=1}{\sum }}\eta
^{l}(\xi _{k})h(Z,Y)=0
\end{equation*}%
\begin{equation*}
h(Y,\nabla _{Z}\xi _{k})+h(Z,Y)=0
\end{equation*}%
\begin{equation*}
h(Y,Z)=0.
\end{equation*}
\end{proof}

\end{document}